\documentclass[12pt]{article}

\usepackage{graphicx,tikz,color,url}
\usepackage{amsmath,amssymb,latexsym}
\usepackage{pgfplots}
\usepackage{float}
\usepackage{mathrsfs}
\usepackage{cite}
\usepackage{soul}
\usetikzlibrary{arrows}
\usepackage{amsthm}
\newtheorem{theorem}{Theorem}[section]

\newtheorem{lemma}[theorem]{Lemma}
\newtheorem{proposition}[theorem]{Proposition}
\newtheorem{corollary}[theorem]{Corollary}

\newtheorem{observation}[theorem]{Observation}
\newtheorem{remark}[theorem]{Remark}

\newcommand{\smallqed}{{\tiny ($\Box$)}}
\newcommand{\cS}{{\cal S}}
\begin{document}
\title{S-packing chromatic critical paths and cycles}

\author{
\and G\"{u}lnaz Boruzanl{\i} Ekinci $^{a,}$\thanks{Email: \texttt{gulnaz.boruzanli@ege.edu.tr}} \\
     \and Csilla Bujt\'{a}s $^{b,c,}$\thanks{Email: \texttt{csilla.bujtas@fmf.uni-lj.si}}
      \and Didem G\"{o}z\"{u}pek $^{d,}$\thanks{Email: \texttt{didem.gozupek@gtu.edu.tr}}
      \and Asl{\i}han G\"{u}r $^{e,}$\thanks{Email: \texttt{agur@gtu.edu.tr}}
	}
\maketitle

\begin{center}
$^a$ Department of Mathematics, Faculty of Science, Ege University, İzmir, T\"{u}rkiye\\
	\medskip
	$^b$ Faculty of Mathematics and Physics, University of Ljubljana, Slovenia\\
	\medskip
	$^c$ Institute of Mathematics, Physics and Mechanics, Ljubljana, Slovenia\\
	\medskip
    $^d$ Department of Computer Engineering, Gebze Technical University, T\"{u}rkiye \\
	\medskip
    $^e$ Department of Mathematics, Gebze Technical University, T\"{u}rkiye \\
	\medskip
 \end{center}

 \begin{abstract}
Let $S=(s_1,s_2,\ldots)$ be a non-decreasing sequence of positive integers. For a graph $G$ with vertex set $V(G)$, a labeling 
$\phi \colon V(G)\to \{1,\ldots,k\}$ is an $S$-packing $k$-coloring 
if, whenever two distinct vertices $u,v\in V(G)$ are assigned the same color $i$, their distance in $G$ is greater than $s_i$. 
The minimum $k$ for which $G$ admits such a coloring is the 
 $S$-packing chromatic number of $G$. A graph $G$ is $\chi_S$-vertex-critical if $\chi_S(G-v) < \chi_S(G)$ for every $v \in V(G)$, and it is $\chi_S$-critical if $\chi_S(H) < \chi_S(G)$ holds for every proper subgraph $H$ of $G$.
In this paper, the exact value of $\chi_S(P_n)$ is determined for every path of order $n$ and for every packing sequence $S$ where $s_i < 2^i$ holds for each entry $s_i$. As a consequence, $\chi_S$-critical and $\chi_S$-vertex-critical paths are identified for each such sequence $S$.
In addition, we extend earlier results on $\chi_S$-critical cycles and provide a complete characterization of $\chi_S$-critical and $\chi_S$-vertex-critical cycles for packing sequences $S= (1, s_2, \dots )$ with $s_2 \in \{2,3\}$ and $s_3,s_4 \in \{4,5,6,7\}$.

 \end{abstract}

\noindent
{\bf AMS Subj.\ Class.\ (2020)}: 05C15, 05C12

\section{Introduction} \label{sec:intro}

 A \emph{packing sequence} $S=(s_1,s_2,\ldots)$ is a non-decreasing sequence of positive integers. Consider a graph $G$ with vertex set $V(G)$. A labeling 
$\phi \colon V(G)\to [k]=\{1,\ldots,k\}$ is an {\em $S$-packing $k$-coloring} 
if, whenever two distinct vertices $u,v\in V(G)$ are assigned the same color $i$, their distance in $G$ satisfies $d_G(u,v)>s_i$. 
The minimum $k$ for which $G$ admits such a coloring is the 
{\em $S$-packing chromatic number} of $G$, denoted by $\chi_S(G)$~\cite{goddard-2012}.

The notion of $S$-packing coloring is very general. In particular, when 
$S=(k,k,\ldots)$, one obtains the class of \emph{$k$-distance colorings}~\cite{kramer-2008}, where $k=1$ coincides with the  classical graph coloring. 
Furthermore, taking $S=(1,2,3,\ldots)$ yields the standard packing coloring ~\cite{Goddard, Bresar2}. In recent years, the concept of $S$-packing coloring has drawn considerable attention~\cite{bresar-2025, bresar-2025b, elzain-2025+, holub-2023, kostochka-2021, liu-2020, mortada-2024, mortada-2025, mortada-2026, yang-2023}. 


Several works in the literature have investigated criticality in packing colorings~\cite{klavzar-2019, bresar-2022, ferme-2022}. As packing colorings have been generalized to $S$-packing colorings, the notion of criticality has also been extended to $S$-packing colorings. In particular, vertex criticality in $S$-packing coloring was introduced by Holub et al.~\cite{holub-2020}, where a graph $G$ is \emph{$\chi_S$-vertex-critical} if $\chi_S(G-v) < \chi_S(G)$ for every $v \in V(G)$. Later, a characterization of $\chi_S$-vertex-critical graphs with $\chi_S(G)=4$ for a family of packing sequences was provided by Klav\v{z}ar et al.~\cite{klavzar-2023}. Recently, criticality in $S$-packing coloring was introduced in~\cite{ekinci-2026}, where a graph $G$ is called \emph{$\chi_S$-critical} if $\chi_S(H) < \chi_S(G)$ 
for every proper subgraph $H$ of $G$. In the same paper, $\chi_S$-critical cycles were characterized under different conditions, 
and the remaining cases were left as an open problem. In this paper, among other results, we address this problem and obtain characterizations of $\chi_S$-critical cycles for several cases.

\subsection{Terminology}
Let  $G=(V(G), E(G))$ be a graph, where $V(G)$ and $E(G)$ correspond to the vertex and edge set of $G$, respectively. The distance between two vertices $u$ and $v$ in  $G$ is denoted by $d_G(u,v)$.
 Let $P_n$ and $C_n$ stand for the path and the cycle of order $n$, respectively.

The set of all packing sequences is denoted by $\mathcal{S}= \{(s_1, s_2, ...): s_1 \le s_2 \le ...\}$. Throughout this paper, we assume that all packing sequences are infinite. The notation $\cS_{a_1,\ldots,a_{i-1},\,\overline{a_i},\,a_{i+1},\ldots}$ 
stands for the set of packing sequences $(s_1,s_2,\ldots)$ satisfying $s_j=a_j$ for $j\ne i$ 
and $a_i \le s_i \le a_{i+1}$. For example, ${\cal S}_{1,2,\overline{4},7}$ is the set of packing sequences with $s_1=1$, $s_2=2$, $4\le s_3\le 7$, and $s_4=7$. Note that $\cS_1$ is the set of all packing sequences $(s_1,s_2,\ldots)$ with $s_1=1$. 


Next, we introduce the following family of packing sequences. For an integer $k \ge 2$, let us define
\begin{equation} \label{eq:def-S-k}
    \cS(k)= \{(s_1, s_2, \dots ) \colon  2^{i-1} \leq s_i <2^i \mbox{ for all } i \leq k-1 \mbox{ and } s_k < 2^{k-1} \}.
    \end{equation}
    Then $\cS(1)$ is not defined, and $\cS(2)= \cS_{1,1}$. In general $\cS(k) \subseteq \cS_1$. 
\medskip 

The following observation follows directly from the definition and will be used in the sequel.

\begin{observation} \label{obs:basic}
   If $S=(s_1, s_2, \dots )$ and $S'=(s_1', s_2', \dots )$ are two packing sequences such that $s_i \leq s_i'$ for every positive integer $i$, then $\chi_S(G) \leq \chi_{S'}(G)$ holds for every graph $G$.
\end{observation}
\subsection{Summary of results}

The earlier studies left open the problem of determining the $S$-packing chromatic number of paths except when $S$ belongs to a very specific subfamily of $\cS$. In~\cite{ekinci-2026}, this question was posed directly. In this paper, we determine $\chi_S(P_n)$ for all packing sequences $S=(s_1,s_2, \dots)$ where the entries are not too large. More precisely, we consider all sequences where $s_i < 2^i$ holds for each  $i \leq \lfloor \log_2(n) \rfloor +1$. Theorems~\ref{thm:path} and~\ref{cor:path-lower-bound}~(ii) can be summarized in the following way.
If $S=(s_1, s_2, \dots)$ and $2^{i-1} \leq s_i < 2^i$ holds for each  $ i \leq \lfloor \log_2(n) \rfloor +1$, then 
$$\chi_S(P_n)= \lfloor \log_2(n) \rfloor +1.
$$
Otherwise, $S \in \cS(k)$ for an integer $k \leq \lfloor \log_2(n) \rfloor +1$, and
$ \chi_S(P_n)= k$.

Based on these results, we identify  $\chi_S$-critical and $\chi_S$-vertex-critical paths for the considered packing sequences. In particular, we show that  if $S \in \cS(k)$, then a path $P_n$ is $\chi_S$-critical (and $\chi_S$-vertex-critical) if and only if $n \in \{2^0, 2^1, \dots, 2^{k-1}\}$.
\medskip

In~\cite{ekinci-2026}, $\chi_S$-critical cycles were identified for the following families of packing sequences.
\begin{theorem}{\rm \cite{ekinci-2026}}
\label{thm1}
	If $n \ge 3$, then the following hold.
	\begin{itemize}
		\item[(i)] If $S \in \mathcal{S}_{1,1}$, then $C_n$ is $\chi_S$-critical if and only if $n$ is odd.
		
		\item[(ii)] If $S \in \mathcal{S}_{1,2,2}$, then $C_n$ is $\chi_S$-critical if and only if $n \in \{3,5\}$.
		
		\item[(iii)] If $S \in \mathcal{S}_{1,\overline{2},3}$, then $C_n$ is $\chi_S$-critical if and only if $n \not\equiv 0 \pmod{4}$.
	\end{itemize}
\end{theorem}
Observe that Theorem~\ref{thm1} characterizes $\chi_S$-critical cycles for $S \in \cS(2)\cup \cS(3)$. This paper concentrates on the next case and solves the problem for every $S \in \cS(4)$.
The seven cases in Theorems~\ref{thm-4.1}, \ref{thm-4.2}, and~\ref{thm-4.3} together provide a complete characterization of $\chi_S$-critical  cycles for $S \in \cS(4)$. For example, it is proved that for every $S \in \mathcal{S}_{1,2,\overline{4},5}$,  $C_n$ is $\chi_S$-critical if and only if $n \in \{3, 5, 6, 7, 9, 10, 11,\allowbreak 17\}$; while for a packing sequence $S$ from $\mathcal{S}_{1,2,\overline{4},7} \cup \cS_{1,3,\overline{4},7} \cup \cS_{1,3,\overline{5}, 6}$, a cycle $C_n$ is $\chi_S$-critical if and only if   $n \not\equiv 0 \pmod{8}$ and $n \neq 4$. In addition, Theorem~\ref{thm:vertexcritical} identifies $\chi_S$-vertex-critical cycles for every $S \in \cS(4)$.

\paragraph{Structure of the paper.} 
In Section~\ref{sec:preliminaries}, several useful lemmas are proved.
Section~\ref{sec:paths} is devoted to the $S$-packing chromatic number of paths, where we establish general results and identify all $\chi_S$-critical and $\chi_S$-vertex-critical paths if $S \in \bigcup_{j=2}^\infty \cS(k)$. Sections~\ref{sec:cycles} and~\ref{sec:vertex-critical-cycles} focus on $\chi_S$-critical cycles and $\chi_S$-vertex-critical cycles over $\cS(4)$, respectively. 

\section{Preliminaries} \label{sec:preliminaries}

 In this section, several auxiliary results are established to provide insight into the problem and to be used in the sequel.

\begin{lemma} \label{lem:path1}
    The following statements are true for every positive integer $n$ and packing sequence $S \in \cS_1$.
    \begin{itemize}
        \item[(i)] There exists an $S$-packing coloring of the path $P_n\colon v_1 \dots v_n$ with $\chi_S(P_n)$ colors such that color~$1$ is assigned to every vertex $v_i$ with an odd index $i$.
         \item[(ii)] If  $S=(1,s_2, \dots)$ and $S'=(s_1', s_2', \dots )$ are two packing sequences such that $s_i'=\lfloor \frac{s_{i+1}}{2} \rfloor$ for every positive integer $i$, then it holds for every $n\ge 2$ that
         \begin{equation} \label{eq:path-half}
             \chi_S(P_n)=\chi_{S'}(P_{\lfloor \frac{n}{2} \rfloor} )+1.
         \end{equation} 
    \end{itemize}

\end{lemma}
\begin{proof}
    (i) Set $S=(1, s_2, \dots)$ and let $\phi$ be an $S$-packing coloring of $G=P_n$. Construct the path $G'$ by removing all vertices with color~$1$ from $G$ and forming a path $G': v_{i_1}\dots v_{i_p}$ with the remaining vertices such that $i_1 < \cdots < i_p$. Then, to get a path $G''$ of order $n$ and a coloring $\phi''$, we insert new vertices with color~$1$ into $G'$. Formally, we    
    define $V(G'')=\{u_1,\dots , u_n\}$ such that $u_j= v_{i_{j/2}}$ holds for every even index $j \leq n$ and $u_1, u_3, \dots$ are new vertices. The remaining vertices of $G'$ (if they exist) are omitted from $G''$.  The edges of $G''$ are chosen so that $G''\colon u_1 \dots u_n$ is a path. Finally, a vertex coloring $\phi''$ of $G''$ is defined such that 
    $$
    \phi''(u_j)= 
    \begin{cases}
	1, &\text{if $j$ is odd},\\
	\phi(v_{i_{j/2}}), &\text{if $j$ is even}.\\
    \end{cases}
    $$
    We state that $\phi''$ is an $S$-packing coloring of $G''$. The $S$-packing coloring condition clearly holds for the vertices assigned color~$1$. If $\phi''(u_x)= \phi''(u_y)= \ell \neq 1$, then $x$ and $y$ are even integers, and $u_x=v_{i_{x/2}}$, $u_y=v_{i_{y/2}}$ hold. Further, the distance of $v_{i_{x/2}}$ and $v_{i_{y/2}}$ is $d'=|\frac{x}{2} - \frac{y}{2}|$ in $G'$. It follows by the construction of $G'$ that $G$ contains at most $d'$ vertices with color $1$ between these two vertices and consequently $$d_G(v_{i_{x/2}},v_{i_{y/2}}) \leq 2d' =|x-y|. $$
    The definition of $\phi''$ gives $\phi(v_{i_{x/2}})=\phi(v_{i_{y/2}})=\ell $, so the distance of these two vertices in $G$ is at least $s_\ell+1$. These facts together imply
     $$ d_{G''}(u_x,u_y) = |x-y| \ge d_G(v_{i_{x/2}},v_{i_{y/2}}) \ge s_\ell +1.
    $$
    This proves that $\phi''$ is an $S$-packing coloring of the path $P_n$ and $\phi''$ does not use more colors than $\phi$. Therefore, if $\phi$ uses $\chi_S(P_n)$ colors, so does $\phi''$. This completes the proof of (i). 
    \medskip

    (ii) According to statement (i), we may start with an $S$-packing coloring $\phi$ of $G=P_n$ where every vertex with an odd index receives the color~$1$ and $\phi$ uses $k=\chi_S(P_n)$ colors. When we construct the path $G'$ by removing the vertices with color~$1$ from $G$, only the vertices of even indices remain. Therefore, $G': v_{i_1}\dots v_{i_p}$ is a path on $p=n'= \lfloor \frac{n}{2} \rfloor$ vertices and $i_j=2j$ for every $j \in [n']$. 
    
   Define the coloring $\phi'$ of $G'$ such that $\phi'(v_{i_j})= \phi(v_{2j})-1$ for every $j \in [n']$.  Then $\phi'$ uses $k-1$ colors. We prove that $\phi'$ is an $S'$-packing coloring of $G'$. Let $\phi'(v_{i_a})= \phi'(v_{i_b})= \ell$ and assume that $a <b$. Hence $\phi(v_{2a})= \phi(v_{2b})= \ell+1$ also holds. These facts together with $s_{\ell+1}\ge 2s_\ell'$ imply
    $$ d_{G'}(v_{i_a}, v_{i_b})=\frac{ d_{G}(v_{2a}, v_{2b})}{2} =b-a \ge \left\lceil \frac{ s_{\ell+1}+1}{2} \right\rceil \ge
    \left\lceil \frac{2 s_{\ell}'+1}{2} \right\rceil =  s_\ell'+1, 
    $$
    and we may infer that $\phi'$ is an $S'$-packing coloring of $G'$. Therefore, $\chi_S(P_n)\ge \chi_{S'}(P_{\lfloor \frac{n}{2} \rfloor} )+1$.

    To prove the other direction, we start with an $S'$-packing coloring $\varphi'$ of $G'= P_{n'} \colon u_1 \dots u_{n'}$ where $n'= \lfloor \frac{n}{2} \rfloor$. We assume that $\varphi'$ uses $k'=\chi_{S'}(G')$ colors. Let $G=P_n \colon v_1\dots v_n$ and let $\varphi$ be a vertex coloring of $G$ such that $\varphi(v_i)=1$ if $i$ is odd, and $\varphi(v_i)=\varphi'(u_{i/2})+1$ if $i$ is even. Then $\varphi$ uses $k'+1$ colors. Now, we show that $\varphi$ is an $S$-packing coloring of $G$. If $\varphi(v_i)=\varphi(v_j)=1$, then $d_G(v_i, v_j) \ge 2$ by definition. If $\varphi(v_i)=\varphi(v_j)=\ell$ for an integer $\ell\ge 2$, then $i$ and $j$ are even, and $\varphi'(u_{i/2})=\varphi'(u_{j/2})= \ell-1$. For this case, $d_G(v_i, v_j)=2 d_{G'}(u_{i/2}, u_{j/2}) $ also holds, and we may infer
    $$ d_G(v_i, v_j)=2 d_{G'}(u_{i/2}, u_{j/2}) \ge 2 (s'_{\ell-1}+1) = 2 \left(\left\lfloor \frac{s_\ell}{2} \right\rfloor +1 \right)\ge s_\ell +1. 
    $$
    Since this inequality holds for every two vertices sharing a color $\ell$ different from~$1$ in $G$, we conclude that $\varphi$ is an $S$-packing coloring of $G$ and therefore $\chi_S(P_n) \leq \chi_{S'}(P_{\lfloor \frac{n}{2} \rfloor} )+1$. This completes the proof for (ii).  
\end{proof}

\begin{lemma} \label{lem:path2}
    The following statements hold for every positive integer $n$ and packing sequence $S=(1, s_2, \dots )$.
    \begin{itemize}
        \item[(i)] If $n$ is even, then $\chi_S(P_{n+1})=\chi_S(P_n)$.
        \item[(ii)] If $S'=(1, s_2', \dots )$ such that $s_i' \in \{2 \lceil \frac{s_i+1}{2}\rceil-1, s_i \}$ for every $i \ge 2$, then $\chi_{S'}(P_n)= \chi_S(P_n)$.
    \end{itemize}    
\end{lemma}
\begin{proof}  \begin{itemize}
        \item[(i)]
     Let $S \in \cS_1$, $P_n \colon v_1\dots v_n$, and $\chi_S(P_n)=k$. By Lemma~\ref{lem:path1} (i), there is an $S$-packing $k$-coloring $\phi$ of $P_n$ that assigns color~$1$ to every vertex $v_i$ with an odd index $i$. If $n$ is even, then $\phi(v_n) \neq 1$, and $\phi$ can easily be extended to an $S$-packing $k$-coloring of $P_{n+1}$. Then $\chi_S(P_{n+1}) \leq \chi_S(P_n)$. As the reverse inequality holds by definition, $\chi_S(P_{n+1})=\chi_S(P_n)$ follows.

   \item[(ii)] If $s_i$ is odd, $2 \lceil \frac{s_i+1}{2}\rceil-1=s_i$, and $s_i'=s_i$. If $s_i$ is even, then $2 \lceil \frac{s_i+1}{2}\rceil-1=s_i+1$ and $s_i'$ equals either $s_i$ or $s_i+1$. 
   Let $\phi$ be an $S$-packing coloring of $P_n \colon v_1\dots v_n$ such that every vertex with an odd index receives color~$1$. Such a coloring with $\chi_S(P_n)$ colors exists by Lemma~\ref{lem:path1} (i). Then the distance between any two vertices $v_x$ and $v_y$ sharing a color $\ell\neq 1$ is even. Therefore, if $s_\ell$ is even, then 
   $$d (v_x,v_y) \ge s_\ell +2  \ge s_\ell' +1; 
   $$
   and if $s_\ell$ is odd, then $s_\ell'=s_\ell$ directly implies $d (v_x,v_y)  \ge s_\ell' +1$. We conclude that $\phi$ is an $S'$-packing coloring and $\chi_S(P_n) \ge \chi_{S'}(P_n)$. The equality then follows by Observation~\ref{obs:basic}.
    \end{itemize}
\end{proof}

\begin{lemma}  \label{lem:cycle}
    Let $S=(1,s_2, \dots)$ and $S'=(s_1', s_2' \dots )$ be packing sequences such that $s_i'=\lfloor \frac{s_{i+1}}{2} \rfloor$ for every positive integer $i$. If $n \ge 3$ is an integer, the following statements hold.
    \begin{itemize}
        \item[(i)] $\chi_S(C_{2n}) \leq \chi_{S'}(C_n) +1 $.
        \item[(ii)] If the cycle $C_{2n}$ is $\chi_S$-critical, then $C_n$ is $\chi_{S'}$-critical.
    \end{itemize}
    \end{lemma}
\begin{proof}
\begin{itemize}

    \item[(i)] Consider an $S'$-packing coloring $\phi'$ of $C_n$ with $k'= \chi_{S'}(C_n)$ colors. We construct a vertex coloring $\phi$ for $C_{2n}$ as in the proof of Lemma~\ref{lem:path1}~(ii). That is, let $G'=C_n\colon u_1\dots u_nu_1$ and $G=C_{2n} \colon v_1 \dots v_{2n}v_1$. If $i$ is odd, let $\phi(v_i)=1$. If $i$ is even, set $\phi(v_i)=\phi'(u_{i/2}) +1$. We state that $\phi$ is an $S$-packing coloring of $G$. The coloring condition clearly holds for the vertices with odd indices. If $i$ and $j$ are even integers and $\phi(v_i)=\phi(v_j)=\ell$, then $\phi'(u_{i/2})= \phi'(u_{j/2})= \ell-1$. Then
    $$d_G(v_i, v_j)= 2 d_{G'}(u_{i/2}, u_{j/2}) \ge 2 (s'_{\ell-1}+1) =  2 \left(\left\lfloor \frac{s_\ell}{2} \right\rfloor +1 \right)\ge s_\ell +1
    $$
    verifies that $\phi$ is an $S$-packing coloring of $G$ that uses $k'+1$ colors. Therefore, $\chi_S(C_{2n}) \leq \chi_{S'}(C_n)+1$. 
    \medskip

    \item[(ii)]  By Lemma~\ref{lem:path1} (ii), $\chi_S(P_{2n})=\chi_{S'}(P_n)+1$. If $C_{2n}$ is $\chi_S$-critical, then
    $$\chi_{S'}(P_n)= \chi_S(P_{2n}) -1 < \chi_S(C_{2n}) -1 \le \chi_{S'}(C_n).
    $$
    Therefore, $C_n$ is $\chi_{S'}$-critical as stated.
    \end{itemize}
\end{proof}

\section{$S$-packing chromatic number and criticality of paths} \label{sec:paths}
In this section we study the $S$-packing chromatic number and $\chi_S$-criticality of paths. In particular, we give exact formulas for $\chi_S(P_n)$ for every $n$ and $S \in \bigcup_{k=2}^\infty \cS(k)$, and identify all $\chi_S$-critical paths when $S$ belongs to this class. The section is concluded with the characterization of $\chi_S$-vertex-critical paths for each sequence $S \in \bigcup_{k=2}^\infty \cS(k)$. 

\begin{theorem}  \label{thm:path}
    Let $k \ge 2$ be an integer and $P_n$ be a path. If $S \in \cS(k)$, then 
    \begin{equation*} 
         \chi_S(P_n)= \min \{k,\lfloor \log_2(n) \rfloor +1\}.
    \end{equation*}   
    In particular, if $n \ge 2^{k-1}$ also holds, then $\chi_S(P_n)=k$. 
\end{theorem}
\begin{proof} Let $S=(1, s_2, \dots)$ be a packing sequence from $\cS(k)$. 
    The proof proceeds by induction on $k$. If $k=2$, then $S \in \cS_{1,1}$ and $ \chi_S(P_n)= 2$ if $n\ge 2$ and $ \chi_S(P_1)= 1$. We may now assume that $k \ge 3$.

    We set 
    $$ S'=\left( \left\lfloor \frac{s_2}{2} \right\rfloor, \left\lfloor \frac{s_3}{2} \right\rfloor, \dots \right) \quad \mbox{and} \quad n'= \left\lfloor \frac{n}{2} \right\rfloor.
    $$
    The condition  on  $s_2$  gives $2 \le s_2 <4$  and hence $\lfloor \frac{s_2}{2} \rfloor =1$. We infer that $S' \in \cS_1$. Similarly,  every entry $s_i'=\lfloor \frac{s_{i+1}}{2} \rfloor$ with $2 \le i \le k-2$ in $S'$ satisfies $2^{i-1} \leq s_i' < 2^i$, and $s_{k-1}' <2^{k-2}$. Then $S' \in \cS(k-1)$. Let $k'=k-1$. By the hypothesis, $\chi_{S'}(P_{n'})= \min \{k', \lfloor \log_2(n') \rfloor +1 \}$.       By Lemma~\ref{lem:path1}~(ii), $\chi_{S}(P_{n})= \chi_{S'}(P_{n'})+1$. Using also the definitions of $n'$ and $k'$, we conclude that
    \begin{align*}
        \chi_{S}(P_{n}) = \chi_{S'}(P_{n'})+1 &=\min \{k', \lfloor \log_2(n') \rfloor +1 \}+1\\
        &=  \min \{k-1, \lfloor \log_2(n)-1 \rfloor +1 \}+1\\
        &= \min \{k, \lfloor \log_2(n) \rfloor +1\}.
    \end{align*}
    This finishes the proof of the theorem.
\end{proof}
\begin{remark} 
       If $S \in \cS(k)$, then for every path $P_n\colon v_1 \dots v_n$, an $S$-packing coloring with $\chi_S(P_n)$ colors can be obtained in the following way. 
        \begin{equation} \label{eq:coloring}
            \phi(v_i)=     \begin{cases} 	j, &\text{if $ i \equiv 2^{j-1} \pmod{2^{j}}$ and $1 \leq j \leq k-1$},\\	
        k, &\text{if $ i \equiv 0 \pmod{2^{k-1}}$.}\\    \end{cases}  
        \end{equation}     
         Note that the color of a vertex $v_i$ can be easily identified if we consider the last non-zero digit in the binary form of $i$. It also shows that every vertex of $P_n$ is colored.  Further, for any two vertices sharing a color $j \leq k-1$, their distance is at least $2^j \ge s_j+1$; and if both vertices are colored with $k$, the distance is  at least $2^{k-1} \ge s_k+1$. If $n \ge 2^{k-1}$, then $\phi $ uses $k=\chi_S(P_n)$ colors; otherwise, the number of colors is $\lfloor \log_2(n) \rfloor +1=\chi_S(P_n)$.
     \end{remark}
     
For packing sequences not contained in $\bigcup_{k=2}^\infty \cS(k)$, we present the following results. Note that Theorem~\ref{thm:path} and Theorem~\ref{cor:path-lower-bound} (ii) together determine $\chi_S(P_n)$ for every $S=(1, s_2, \dots )$ satisfying $s_i <2^i$ for all $i \ge 2$.
\begin{theorem}
   \label{cor:path-lower-bound}
Let  $S=(1, s_2, \dots )$ be a packing sequence and $P_n$ be a path. 
\begin{itemize}
    \item[(i)] If there exists an index $i$ such that $s_i < 2^{i-1}$ and $k= \min\{ i \colon s_i < 2^{i-1}\}$, then 
   $$ \chi_S(P_n)\ge  \min \{k,\lfloor \log_2(n) \rfloor +1\}.
   $$    
   \item[(ii)] If $2^{i-1} \leq s_i < 2^i$ holds for every  $2 \le i \le \lfloor \log_2(n) \rfloor +1$, then $$ \chi_S(P_n)=  \lfloor \log_2(n) \rfloor +1.
   $$
   \item[(iii)] If $2^{i-1} \leq s_i $ holds for every  index $2 \le i \le \lfloor \log_2(n) \rfloor +1$, then $$ \chi_S(P_n)\ge   \lfloor \log_2(n) \rfloor +1.
   $$ 
   \end{itemize} 
   \end{theorem} 
   \begin{proof}  
   \begin{itemize}
     \item[(i)] Let $S=(1,s_2, \dots)$ be a packing sequence that satisfies the conditions in (i), and define  $S'=(1,s_2', \dots) $ such that $s_i'= 2^{i-1}$ holds for every $1 \le i \le k-1$ and $s_j'= s_j$ for all $j \ge k$. In particular, $s_p \ge s_p'$ holds for every positive integer $p$. Observation~\ref{obs:basic} therefore implies $\chi_S(P_n) \ge \chi_{S'} (P_n)$. Since $S' \in \cS(k)$, Theorem~\ref{thm:path} gives $\chi_{S'}(P_n)=  \min \{k,\lfloor \log_2(n) \rfloor +1\}$ and statement (i) follows.

       \item[(ii)] Set $k= \lfloor \log_2(n) \rfloor +1$. Then $2^{k-1} \leq n < 2^{k}$.  First observe that the coloring~\eqref{eq:coloring} applied for the path $P_n\colon v_1 \dots v_n$ is an $S$-packing coloring of $P_n$ which uses $k$ colors. Consequently, $\chi_S(P_n) \leq k$.  On the other hand, define the packing sequence $S''= (1, s_2'', \dots)$ such that $s_i''=s_i$ for all indices except that $s_k''=s_{k-1}$. Hence $s_i'' \leq s_i$ holds for every $i$ and Observation~\ref{obs:basic} implies $\chi_{S''}(P_n) \leq \chi_{S}(P_n)$. By the conditions in (ii), Theorem~\ref{thm:path} can be applied to $S''$, $P_n$, and $k$. This yields 
        $$ \chi_{S''}(P_n)=  \min \{k,\lfloor \log_2(n) \rfloor +1\}. $$  
      Since $k= \lfloor \log_2(n) \rfloor +1$, we may derive 
      $$ \chi_{S}(P_n) \ge \chi_{S''}(P_n) =\lfloor \log_2(n) \rfloor +1 =k,
      $$
      which completes the proof for the equality in (ii).

       \item[(iii)] Under the conditions of (iii), we may define a packing sequence $S^*=(1, s_2^*,\dots) $, where $s_i^*=2^{i-1}$ for every index $2 \le i \le \lfloor \log_2(n) \rfloor +1$, and $s_i^*=s_i$ for $i > \lfloor \log_2(n) \rfloor +1$. By the equality stated in (ii), $\chi_{S^*}(P_n)  =\lfloor \log_2(n) \rfloor +1$. As $s_i^* \leq s_i$ for all indices, we may deduce statement~(iii) from Observation~\ref{obs:basic}.
   \end{itemize}
   \end{proof}

   Having Theorem~\ref{thm:path} and Theorem~\ref{cor:path-lower-bound} (ii) in hand, we are now able to identify the $\chi_S$-critical paths  for every packing sequence $S=(1,s_2, \dots )$ satisfying (i) $S \in \bigcup_{k=2}^\infty \cS(k)$ or (ii) $s_i < 2^i$  for each $i \ge 2$. 
 
 \begin{theorem} \label{thm:critical-path} Let $S=(1,s_2, \dots)$ be a packing sequence and $k \ge 2$ be an integer. 
 \begin{itemize}
     \item[(i)] If $S \in \cS(k)$, then a path $P_n$ is $\chi_S$-critical if and only if $n \in \{2^0, 2^1, \dots, 2^{k-1}\}$.
     \item[(ii)] If $2^{i-1} \le s_i < 2^i$ holds for every entry $s_i$, then a path $P_n$ is $\chi_S$-critical if and only if $n \in \{2^j \colon  j \in \mathbb{N}_{0}\}$.
 \end{itemize}
       
   \end{theorem}
   \begin{proof}
       By definition, $P_1$ is $\chi_S$-critical for every $S \in \cS$. From now on, we consider only paths of order at least two.    
    \begin{itemize}
    \item[(i)] Let $S \in \cS(k)$ and $n \ge 2$. Assume first that the path $G=P_n\colon v_1\dots v_n$ is $\chi_S$-critical. By removing the edge $v_1v_2$ from $G$ we obtain $G-v_1v_2$, which is the disjoint union of an isolated vertex and a path of order $n-1$. From the $\chi_S$-criticality of $P_n$  we may infer
       $$\chi_S(P_n) > \chi_S(G-v_1v_2) = \max\{ \chi_S(P_1), \chi_S(P_{n-1})\} = \chi_S(P_{n-1}).
       $$
       By Theorem~\ref{thm:path}, the strict inequality holds if and only if $n \in \{2^1, 2^2, \dots, 2^{k-1}\}$. We conclude that it is a necessary condition for the $\chi_S$-criticality of $P_n$ if $n \ge 2$.

       Now assume that $n=2^{i}$, for an integer $1 \le i \le k-1$. By Theorem~\ref{thm:path}, $\chi_S(P_n)= i+1$. The removal of an edge $e$ from $G=P_n$ yields $G-e$, which consists of two paths $P_r$ and $P_q$ such that $r+q= n$. We may suppose that $r \le q$ and hence,  $\chi_S(G-e)= \chi_S(P_q)$. Observe that $q \leq 2^i-1$.
       Using the formula in  Theorem~\ref{thm:path} again, we obtain
       $$\chi(G-e)= \chi_S(P_q) < i+1 =  \chi_S(P_n)
       $$
        that shows the $\chi_S$-criticality of $P_n$, and the proof is complete for (i).

     \item[(ii)]  If $S=(1,s_2, \dots )$ is a packing sequence satisfying $2^{i-1} \leq s_i < 2^i$ for all $s_i$, we may use the formula $\chi_S(P_n)=  \lfloor \log_2(n) \rfloor +1$ from Theorem~\ref{cor:path-lower-bound} (ii). Then $\chi_S(P_{n-1}) < \chi_S (P_n)$ holds if and only if $n \in \{2^1, 2^2, \dots \}$ as stated. 
     \end{itemize}
   This finishes the proof of the theorem.
    \end{proof}
 Every $\chi_S$-critical graph $G$ is $\chi_S$-vertex-critical by definition. The following proposition shows that the reverse implication also holds if $G$ is a path.
  \begin{proposition} \label{prop:critical-path}
      Let $S$ be a packing sequence and $P_n$ be a path of order $n$. Then $P_n$ is $\chi_S$-vertex-critical if and only if it is $\chi_S$-critical.
  \end{proposition}
  \begin{proof} By definition, it suffices to show that, for every $S \in \cS$ and $n \ge 2$, if  $P_n$ is $\chi_S$-vertex-critical, then it is $\chi_S$-critical.   
  The deletion of a leaf from $P_n$ gives $P_{n-1}$ and the $\chi_S$-vertex-criticality of $P_n$ implies $\chi_S(P_{n-1}) < \chi_S(P_n)$. On the other hand, every proper subgraph $H$ of $P_n$ consists of path components $H_1, \dots, H_\ell$, each of which is of order at most $n-1$. Since $\chi_S(H) = \max \{\chi_S(H_i) \colon 1 \le i \le \ell\}$, the inequality $\chi_S(H) \leq \chi_S(P_{n-1}) $ holds for every proper subgraph $H$. Therefore, $\chi_S(P_{n-1}) < \chi_S(P_n)$ implies $\chi_S(H) < \chi_S(P_n)$ for every $H \subset P_n$, and we may conclude that $P_n$ is $\chi_S$-critical.
  \end{proof}
  Theorem~\ref{thm:critical-path} and Proposition~\ref{prop:critical-path}  directly imply the following statement on the $\chi_S$-vertex-criticality of paths.
  \begin{corollary} \label{cor:vertex-crit-path}
     Let $S=(1,s_2, \dots)$ be a packing sequence and $k \ge 2$ an integer. 
 \begin{itemize}
     \item[(i)] If $S \in \cS(k)$, then a path $P_n$ is $\chi_S$-vertex-critical if and only if $n \in \{2^0, 2^1, \dots, \allowbreak 2^{k-1}\}$.
     \item[(ii)] If $2^{i-1} \le s_i < 2^i$ holds for every entry $s_i$, then a path $P_n$ is $\chi_S$-vertex-critical if and only if $n \in \{2^j \colon  j \in \mathbb{N}_{0}\}$.
 \end{itemize} 
  \end{corollary}

   
\section{Results on $\chi_S$-critical cycles} \label{sec:cycles}

In this section, we characterize the $\chi_S$-critical cycles for each packing sequence $S \in \cS(4)$. Recall that the same problem was solved in~\cite{ekinci-2026} for packing sequences in $\cS(2) \cup \cS(3)$  (see Theorem~\ref{thm1}). 
We consider the case $S \in \mathcal S_{1,2,s_3,s_4}$ with $s_3,s_4 \in \{4,5,6\}$.
\begin{theorem} \label{thm-4.1}
The following statements hold.

    \begin{itemize} 
    \item[$(i)$] If $S \in \mathcal{S}_{1,2,4,4}$, then $C_n$ is $\chi_S$-critical if and only if $n \in \{3, 5, 6, 7, 9\}$.

    \item[$(ii)$] If $S \in \mathcal{S}_{1,2,\overline{4},5}$, then $C_n$ is $\chi_S$-critical if and only if $n \in \{3, 5, 6, 7, 9, 10, 11,\allowbreak 17\}$.
\item[$(iii)$] If $S \in \mathcal{S}_{1,2,\overline{4},6}$, then $C_n$ is $\chi_S$-critical if and only if $n \in \{3, 5, 6, 7, 9, 10, 11,\allowbreak 12, 13, 17, 18, 19, 20, 25, 26, 27, 33, 34, 41\}$.        
    \end{itemize}
\end{theorem}
\begin{proof} 
Let $S \in \mathcal S_{1,2,s_3,s_4}$, where $s_3, s_4 \in \{4, 5, 6\}$. Assume first that $n \ge 8$. Theorem \ref{thm:path} implies that $\chi_S(P_n)=4$ and hence $\chi_S(C_n)\ge 4$. 
    
 For $n \ge 8$ with $n \notin \{9,10,11,12,13,17,18,19,20,25,26,27,33,34,41\}$, we consider the following colorings of $C_n \colon v_1 \dots v_nv_1$. Each coloring starts at vertex $v_1$ and is given by a suitable prefix, followed by repetitions of the pattern $12131214$, chosen so that $v_n$ receives color $4$.
Let $n \equiv r \pmod {8}$. 
If $r=0$, we use the coloring $(12131214)^*$. 
If $1\le r\le 7$, we use the  coloring $(1213124)^{8-r}(12131214)^* .$ Note that such integers  are precisely those of the form $n=7k+8m$ for some $k,m \in \mathbb{N}_{0}$.
Consequently, $\chi_S(C_n)=4$ whenever $n=7k+8m$ with $k,m \in \mathbb{N}_{0}$.

For the remaining values of $n$, we determine $\chi_S(P_n)$ by Theorem~\ref{thm:path} and 
$\chi_S(C_n)$ by direct computation using the tool in \cite{Spackingtool} when $S \in \cS_{1,2,4,4} \cup \cS_{1,2,4,5} \cup \cS_{1,2,4,6} \cup \cS_{1,2,5,5}$; the results are summarized in 
Table~\ref{tab:smallcases}. Observe that $\chi_{S^*}(P_n) < \chi_{S^*}(C_n)$ holds for every value $n$ listed in the table when $S^* \in \cS_{1,2,4,6}$. We infer the same inequality for any packing sequence $S'$ from $\cS_{1,2,\overline{5},6}$. Indeed, we have $\chi_{S^*}(P_n) =\chi_{S'}(P_n)$ by Theorem~\ref{thm:path}, and Observation~\ref{obs:basic} implies $\chi_{S^*}(C_n) \leq \chi_{S'}(C_n)$. Therefore, exactly the same cycles are $\chi_S$-critical for a packing sequence from $\cS_{1,2,\overline{5},6}$ as for those from $\cS_{1,2,4,6}$.
\begin{table}[!ht]
\centering
\begin{tabular}{|c||c||c|c|c|c|}
\hline
$n$ & $\chi_{S_i}(P_n)$ 
& $\chi_{S_1}(C_n)$ 
& $\chi_{S_2}(C_n)$ 
& $\chi_{S_3}(C_n)$ 
& $\chi_{S_4}(C_n)$ \\
\hline
3  & 2 & 3 & 3 & 3 & 3 \\ \hline
4  & 3 & 3 & 3 & 3 & 3 \\ \hline
5  & 3 & 4 & 4 & 4 & 4 \\ \hline
6  & 3 & 4 & 4 & 4 & 4 \\ \hline
7  & 3 & 4 & 4 & 4 & 4 \\ \hline
9  & 4 & 5 & 5 & 5 & 5 \\ \hline
10 & 4 & 4 & 5 & 5 & 5 \\ \hline
11 & 4 & 4 & 5 & 5 & 5 \\ \hline
12 & 4 & 4 & 4 & 5 & 4 \\ \hline
13 & 4 & 4 & 4 & 5 & 4 \\ \hline
17 & 4 & 4 & 5 & 5 & 5 \\ \hline
18 & 4 & 4 & 4 & 5 & 4 \\ \hline
19 & 4 & 4 & 4 & 5 & 4 \\ \hline
20 & 4 & 4 & 4 & 5 & 4 \\ \hline
25 & 4 & 4 & 4 & 5 & 4 \\ \hline
26 & 4 & 4 & 4 & 5 & 4 \\ \hline
27 & 4 & 4 & 4 & 5 & 4 \\ \hline
33 & 4 & 4 & 4 & 5 & 4 \\ \hline
34 & 4 & 4 & 4 & 5 & 4 \\ \hline
41 & 4 & 4 & 4 & 5 & 4 \\
\hline
\end{tabular}
\caption{Values of $\chi_{S_i}(P_n)$ and $\chi_{S_i}(C_n)$ for small values of $n$ when $S_1 \in \mathcal{S}_{1,2,4,4}$, $S_2 \in \mathcal{S}_{1,2,4,5}$, $S_3 \in \mathcal{S}_{1,2,4,6}$, and  $S_4 \in \mathcal{S}_{1,2,5,5}$.}\label{tab:smallcases}
\end{table}

Therefore $\chi_S(P_n)<\chi_S(C_n)$ holds
precisely for the values of $n$ listed in the statement of the theorem. 
\end{proof}

The next theorem characterizes the $\chi_S$-critical cycles for the packing sequences 
$S \in \mathcal S_{1,3,s_3,s_4}$ with $s_3,s_4 \in \{4,5\}$.

\begin{theorem} \label{thm-4.2}
The following statements hold.

	\begin{itemize}
		
		\item[$(i)$] 	If $S \in \mathcal{S}_{1,3,4,4}$, then $C_n$ is $\chi_S$-critical if and only if $n \in \{3, 5, 6, 7, 9\}$.
		
		\item[$(ii)$] If $S \in \mathcal{S}_{1,3,4,5}$, then $C_n$ is $\chi_S$-critical if and only if $n \in \{3, 5, 6, 7, 9, 10, 11, \allowbreak 15,  17, 23\}$.
		
		\item[$(iii)$] 	If $S \in \mathcal{S}_{1,3,5,5}$, then $C_n$ is $\chi_S$-critical if and only if $n \in \{6 ,10\}$ or $n$ is odd.
	\end{itemize}
\end{theorem}

\begin{proof}
	Let $S \in \mathcal S_{1,3,s_3,s_4}$, where $s_3, s_4 \in \{4, 5\}$. Theorem \ref{thm:path} implies that $\chi_S(P_n)=4$ for $n\geq 8$. Thus, we have $\chi_S(C_n)\ge 4$. Let $C_n \colon v_1v_2 \dots v_nv_1$.
	\begin{itemize} 
		
	\item[$(i)$]  
	For  $S \in \mathcal{S}_{1,3,4,4}$ and $n \ge 8$  with $n \notin \{9,11,12,14,17,19,22,27\}$, we consider the following colorings of $C_n$. 
	Let $n \equiv r \pmod {8}$. 
	If $r=0$, we use the coloring $(12131214)^*$. If $r\neq0$ and $n=5k+8m$ for some $k,m \in \mathbb{N}_{0}$, then we set  $(13214)^k(12131214)^m.$ Observe that each value of $n$ considered now can be written in the form 
$n = 5k + 8m$ for some $k, m \in \mathbb{N}_{0}$, and  consequently $\chi_S(C_n)=4$.
	
    For the remaining values,  
	$\chi_S(C_n)$ can be determined by direct computation~\cite{Spackingtool}; the results are summarized in 
	Table~\ref{tab:smallcases1}.

\begin{table}[ht!]

\centering
\begin{tabular}{|c||c||c|c|c|c|}
\hline
$n$ & $\chi_{S_i}(P_n)$ 
& $\chi_{S_1}(C_n)$ 
& $\chi_{S_2}(C_n)$ 
& $\chi_{S_3}(C_n)$  \\
\hline
3  & 2 & 3 & 3 & 3 \\ \hline
4  & 3 & 3 & 3 & 3 \\ \hline
5  & 3 & 4 & 4 & 4 \\ \hline
6  & 3 & 4 & 4 & 4 \\ \hline
7  & 3 & 5 & 5 & 5 \\ \hline
9  & 4 & 5 & 5 & 5\\ \hline
10 & 4 & 4 & 5 & 5 \\ \hline
11 & 4 & 4 & 5 & 6 \\ \hline
12 & 4 & 4 & 4 & 4 \\ \hline
14 & 4 & 4 & 4 & 4 \\ \hline
15 & 4 & 4 & 5 & 5 \\ \hline
17 & 4 & 4 & 5 & 5 \\ \hline
19 & 4 & 4 & 4 & 5 \\ \hline
22 & 4 & 4 & 4 & 4 \\ \hline
23 & 4 & 4 & 5 & 5 \\ \hline
27 & 4 & 4 & 4 & 5 \\ 
\hline
\end{tabular}
\caption{Values of $\chi_{S_i}(P_n)$ and $\chi_{S_i}(C_n)$ for small values of $n$ when $S_1 \in \mathcal{S}_{1,3,4,4}$, $S_2 \in \mathcal{S}_{1,3,4,5}$ and $S_3 \in \mathcal{S}_{1,3,5,5}$.}\label{tab:smallcases1}
\end{table}

		\item[$(ii)$] Let $S \in \mathcal{S}_{1,3,4,5}$. 
		If  $n$ is even with $n \geq 12$, then $n= 6k+8m$ for some $k,m \in \mathbb{N}_{0}$. Thus, we can use the coloring $(131214)^{k}(12131214)^m.$  Consequently, $\chi_S(C_n)=4$ for these cases. Let $n$ be odd and $n \notin \{3, 5, 7, 9, 11, 15, 17, 23\}$. Then $n$ can be written in one of the following forms: (a) $n=7+6m$, (b) $15+6m$, (c) $23+6m$ so that $m \in \mathbb{N}_{0}$  so that $m \in \mathbb{N}^+$. In case of (a) we color $C_n$ as $(1213124)(131214)^m.$ In case of (b) we color $C_n$ as \[ (1213124)(13121412)(131214)^m.\] In case of (c) we color $C_n$ as  \[(1213124)(13121412)^2(131214)^m.\]
        For the remaining values of $n$,  
		$\chi_S(C_n)$ is determined by direct computation \cite{Spackingtool}; the results are summarized in 
		Table~\ref{tab:smallcases1}.   
		
		\item[$(iii)$] 
		Let $S \in \mathcal{S}_{1,3,5,5}$. If  $n$ is even with $n \geq 8$ and $n\neq 10$, then $n= 6k+8m$ for some $k,m \in \mathbb{N}_{0}$. We define the same coloring as in the proof of (ii) and color $C_n$ with the pattern $(131214)^{k}(12131214)^m.$  Consequently, $\chi_S(C_n)=4$ for these cases. On the other hand, Table~\ref{tab:smallcases1} shows that $C_{10}$ is $\chi_S$-critical as $\chi_S(C_{10})=5$.

It remains to show that $\chi_S(C_n)\ge 5$ holds if $n $ is odd and $S \in \cS_{1,3,5,5}$. Suppose to the contrary that $\phi$ is an $S$-packing coloring of $C_n$ with four colors. Let $\phi$ be chosen such that it contains the maximum number of vertices colored with~$1$ among such colorings of $C_n$. We prove four claims to get a contradiction.

Observe that if one vertex receives color $1$ for every pair of consecutive vertices on the cycle, then the colors on the cycle alternate between $1$ and a color different from $1$. This implies that the cycle has even length, contradicting the assumption that $n$ is odd. Hence there exist two adjacent vertices $x$ and $y$ in the cycle such that $\phi(x) \neq 1 \neq \phi(y)$. 

\textbf{Claim 1.} From every three consecutive vertices $x$, $y$, and $z$, at least one is colored with $1$. \\
   \noindent \textit{Proof.}
  Suppose to the contrary that none of the consecutive vertices $x$, $y$, $z$ is assigned color $1$.  In this case, we may recolor the middle vertex $y$ with color~$1$  and obtain an $S$-packing coloring where the number of vertices receiving color~$1$ is greater than in $\phi$. This contradicts the choice of $\phi$ having the maximum possible number of vertices colored with $1$ among such colorings of $C_n$. \smallqed

    \textbf{Claim 2.} There are no two neighbors $x$ and $y$ in the cycle such that $\phi(x)=2$ and $\phi(y)=3$.\\
   \noindent \textit{Proof.}
Suppose for a contradiction that $\phi(x)=2$ and $\phi(y)= 3$. We may assume without loss of generality that $x= v_3$ and $y=v_4$. By Claim~1, $\phi(v_2)=\phi(v_5)=1$. Then, the only possibility for $v_1$ to be colored is $\phi(v_1)=4$. However, $v_6$ cannot be assigned colors $1$, $2$, $3$, and since $d(v_1, v_6) <6$, it also cannot receive color~$4$, a contradiction. \smallqed

\textbf{Claim 3.} There are no two neighbors $x$ and $y$ in the cycle such that $\phi(x)=2$ and $\phi(y)= 4$.\\
   \noindent \textit{Proof.}
 Since colors $3$ and $4$ have the same distance requirements, this statement is equivalent to Claim~2.  \smallqed

\textbf{Claim 4.} There are no two neighbors $x$ and $y$ in the cycle such that $\phi(x)=4$ and $\phi(y)= 3$.\\
   \noindent \textit{Proof.}
Suppose to the contrary that $\phi(v_1)=4$ and $\phi(v_2)= 3$. Again, Claim~1 implies $\phi(v_3)=1$. Then the coloring constraint requires that $\phi(v_4)=2$ and $\phi(v_5)=1$. Furthermore, $v_6$ can get neither of the colors $1$, $2$, $3$, $4$, contradiction. \smallqed

Claims~1-4 exclude the presence of two adjacent vertices $x$ and $y$ satisfying $\phi(x) \neq 1 \neq \phi(y)$,  a contradiction.
Therefore, if $n$ is odd and  $S \in \mathcal{S}_{1,3,5,5}$, then $\chi_S(C_n) \ge 5$. 
\end{itemize}

 For $3 \leq n \leq 7$ and $S \in \mathcal{S}_{1,3,4,4} \cup \cS_{1,3,\overline{4},5}$, the $\chi_S$-criticality of $C_n$ can be checked in Table~\ref{tab:smallcases1}.
\end{proof}

The following theorem characterizes the $\chi_S$-critical cycles for all packing sequences 
from $\mathcal{S}_{1,2,\overline{4},7} \cup \mathcal{S}_{1,3,\overline{4},7} \cup \mathcal{S}_{1,3,\overline{5},6}$.

\begin{theorem}  \label{thm-4.3}
      If $S \in \mathcal{S}_{1,2,\overline{4},7} \cup \cS_{1,3,\overline{4},7} \cup \cS_{1,3,\overline{5}, 6}$, then $C_n$ is  $\chi_S$-critical if and only if   $n \not\equiv 0 \pmod{8}$ and $n \neq 4$.  
    \end{theorem}  
     \begin{proof}
    Let $S \in \mathcal{S}_{1,2,\overline{4},7} \cup \cS_{1,3,\overline{4},7} \cup \cS_{1,3,\overline{5}, 6}$ and $n \ge 8$. By Theorem~\ref{thm:path},  $\chi_S(P_n) =4$.
  If $n \equiv 0 \pmod{8}$, the cycle $C_n$ may be colored by using the pattern $(12131214)^*$ that shows $\chi_S(C_n)\le 4$. Since $P_n$ is a subgraph of $C_n$, the inequality $\chi_S(C_n) \ge \chi_S(P_n)=4$ also holds and proves $\chi_S(C_n)=4$. 
   Consequently, $C_n$ is not $\chi_S$-critical if $n \equiv 0 \pmod{8}$.
 \medskip

  It suffices to show that $\chi_S(C_n)\ge 5$ holds if $n \not\equiv 0 \pmod{8}$. First, we assume that $S \in \cS_{1,2,4,7} \cup \cS_{1,3,5,6}$ and suppose for a contradiction that $\phi$ is an $S$-packing coloring of $C_n$ with $4$ colors. We may further suppose that by $\phi$ the maximum possible number of vertices are assigned color~$1$ among such colorings of $C_n$. The proof continues with a sequence of claims which show that the supposition of $\phi$ using only four colors leads to a contradiction, no matter whether $S \in \cS_{1,2,4,7}$ or $S \in \cS_{1,3,5,6}$. 

 \textbf{Claim 1.}   There are two consecutive vertices in the cycle such that neither of them is assigned  color~$1$. \\
 \noindent \textit{Proof.}
 Suppose that the statement is not true. Then $n$ is even and every second vertex of $C_n$ receives color~$1$. Let $V(C_n)= \{v_1, \dots, v_n\}$ and $\phi(v_i)=1$ for every odd index $i$.
  Similarly to the proof of Lemma~\ref{lem:cycle} (i), consider the cycle $G'=C_{n/2}$ on the vertex set $V(G')=\{v_2, v_4, \dots,  v_{n}\}$ and the coloring $\phi'$ of $G'$  such that $\phi'(v_i)=\phi(v_i)-1$ for every  $v_i \in V(G')$. Let $S'$ be an arbitrary packing sequence from $\cS_{1,2,3}$. 
  
  We prove that $\phi'$ is an $S'$-packing coloring. If $\phi'(v_i)=\phi'(v_j)= 1$ holds for $v_i, v_j \in V'$, then $\phi(v_i)=\phi(v_j)=2$ and, as $s_2\ge 2$, the distance $d$ of $v_i$ and $v_j$ is at least $3$ in $C_n$. In fact, as every second vertex is assigned color~$1$, we have $d \ge 4$. Therefore, the distance of $v_i$ and $v_j$ 
   in $C_{n/2}$ is at least $2$. Similarly, we can derive that $\phi'(v_i)=\phi'(v_j)= 2$ implies that the distance  
     of $v_i$ and $v_j$ is at least $6$ in $C_n$, and 
   at least $3$ in $C_{n/2}$. Moreover,  $\phi'(v_i)=\phi'(v_j)= 3$ implies that the distance is at least $4$. Hence, $\phi'$ is an $S'$-packing coloring of $C_{n/2}$ and $\chi_{S'}(C_{n/2}) \leq 3 $. As $n/2 \ge 4$ and $S' \in \cS(3)$, Theorem~\ref{thm:path} gives $\chi_{S'}(P_{n/2})=3$ showing that  $C_{n/2}$ is not $S'$-packing critical. By Theorem~\ref{thm1} (iii), this implies $n/2 \equiv 0 \pmod{4}$ which is impossible according to our assumption $n \not\equiv 0 \pmod{8} $. \smallqed

  \textbf{Claim 2.} From every three consecutive vertices, at least one receives color~$1$. \\
   \noindent \textit{Proof.}
   The argument is analogous to that in the proof of Theorem~\ref{thm-4.2}.
    That is, supposing that neither of $v_i$, $v_{i+1}$, $v_{i+2}$ is assigned  color~$1$, we may recolor the middle vertex $v_{i+1}$ with color~$1$ and obtain an $S$-packing coloring. This contradicts the choice of $\phi$.  \smallqed
   \medskip
   
Claims~1 and 2 imply that there exist four consecutive vertices, say $v_2$, $v_3$, $v_4$, $v_5$, such that $\phi(v_2)=\phi(v_5)=1$ and $\phi(v_3)\neq 1 \neq \phi(v_4)$.  
  
\textbf{Claim 3.} 
It is not possible that $\phi(v_3)=2$ and $\phi(v_4)= 3$.\\
 \noindent \textit{Proof.}
First we consider the case with $S \in \cS_{1,2,4,7}$. Suppose for a contradiction that $\phi(v_3)=2$ and $\phi(v_4)= 3$. We know that $\phi(v_2)=\phi(v_5)=1$. Then, the only possibility for $v_1$ to be colored is $\phi(v_1)=4$. Similarly, we may conclude that $\phi(v_6)= 2$ and $\phi(v_7)=1$. However, $v_8$ cannot be assigned colors $1$, $2$, $3$, and since $d(v_1, v_8) <8$, it also cannot get color~$4$. 

Now we continue with the remaining case where $S \in \cS_{1,3,5,6}$. Suppose, to the contrary of the statement, that $\phi(v_3)=2$ and $\phi(v_4)= 3$. We know that $\phi(v_2)=\phi(v_5)=1$. Since here $s_2=3$, the only possibility for both $v_1$ and $v_6$ is such that $\phi(v_1)=\phi(v_6)=4$.  However, $v_1$ and $v_6$ cannot both be colored $4$ since $d(v_1,v_6)=5<7$.  
These contradictions prove Claim~3. \smallqed

\textbf{Claim 4.} It is not possible that $\phi(v_3)=4$ and $\phi(v_4)= 2$.\\
 \noindent \textit{Proof.}
Suppose that the claim is not true and $\phi(v_3)=4$, $\phi(v_4)= 2$. Since $\phi(v_2)=\phi(v_5)=1$, $s_1=1$, $s_2 \ge 2$, and $s_4 \ge 6$, the coloring condition implies $\phi(v_6)=3$.

If $S \in \cS_{1,3,5,6}$, then $v_1$ cannot receive color $1$, $2$, or $4$, and hence $\phi(v_1)=3$. 
Since $d(v_1, v_6)=5$, it directly gives the desired contradiction.

If $S \in \cS_{1,2,4,7}$, then $\phi(v_5)=1$ and $\phi(v_6)=3$ remains true and we continue by considering the color of $v_7$. As $\phi(v_3)=4$ and $\phi(v_6)=3$, colors $3$ and $4$ cannot be assigned to $v_7$. By Claim~3, $\phi(v_7) \neq 2$. 
Therefore $\phi(v_7)=1$, and then $\phi(v_8)=2$. The only possibility for $v_9$ is to set $\phi(v_9)=1$. However, neither of the four colors can be assigned to $v_{10}$, a contradiction. \smallqed

\textbf{Claim 5.} It is not possible that $\phi(v_3)=4$ and $\phi(v_4)= 3$.\\
 \noindent \textit{Proof.}
Suppose for a contradiction that $\phi(v_3)=4$, and $\phi(v_4)= 3$. Then  $\phi(v_2)=\phi(v_5)=1$ and, since $s_3 \ge 4$ and $s_4 \ge 6$, the coloring constraint requires that $\phi(v_6)=2$ and $\phi(v_7)=1$. Furthermore, $v_8$ can get neither of the colors $1$, $2$, $3$, $4$, a contradiction. \smallqed
\medskip

Claims 1-5 together give a contradiction when $S \in \cS_{1,2,4,7} \cup \cS_{1,3,5,6}$. Therefore, $\chi_S(C_n) \ge 5$ holds for every $S \in \cS_{1,2,4,7} \cup \cS_{1,3,5,6}$ if $n \ge 9$ and $n \not\equiv 0 \pmod{8}$.
By Observation~\ref{obs:basic}, the same result $\chi_S(C_n) \ge 5$ can be extended to all packing sequences $S \in \mathcal{S}_{1,2,\overline{4},7} \cup \cS_{1,3,\overline{4},7} \cup \cS_{1,3,\overline{5}, 6}$ and the $\chi_S$-criticality follows.
\medskip

For the remaining small cases where $n \le 7$, we observe that
$\chi_S(P_3)=2 < 3= \chi_S(C_3)$,
$\chi_S(P_4)=3=\chi_S(C_4)$,
$\chi_S(P_5)=3 < 4 = \chi_S(C_5)$,
$\chi_S(P_6)=3 < 4 = \chi_S(C_6)$,
$\chi_S(P_7)=3 < 4 \le \chi_S(C_7)$ are true for all $S \in \mathcal{S}_{1,2,\overline{4},7} \cup \cS_{1,3,\overline{4},7} \cup \cS_{1,3,\overline{5}, 6}$. This finishes the proof of the theorem.
    \end{proof}
 
\section{$\chi_S$-vertex-critical cycles} \label{sec:vertex-critical-cycles}
In this section, we characterize the $\chi_S$-vertex-critical cycles for every packing sequence $S \in \cS(4)$.
The result is given in the following theorem.
   \begin{theorem} \label{thm:vertexcritical} Let $S$ be a packing sequence. For a cycle $C_n$ with $n \ge 3$, the following statements hold.

    \begin{itemize} 
    
    \item[$(i)$] If $S \in \mathcal{S}_{1,2,4,4}$ or  $S \in \mathcal{S}_{1,3,4,4}$, then $C_n$ is $\chi_S$-vertex-critical if and only if $n \leq 9$.
    
    \item[$(ii)$] If $S \in \mathcal{S}_{1,2,\overline{4},5}$, then $C_n$ is $\chi_S$-vertex-critical if and only if $n \leq 11$ or $n = 17$.
    
    \item[$(iii)$] If $S \in \mathcal{S}_{1,2,\overline{4},6}$, then $C_n$ is $\chi_S$-vertex-critical if and only if $n \leq 13 $ or $n \in \{17, 18, 19, 20, 25, 26, 27, 33, 34, 41\}$.        
	
	\item[$(iv)$] If $S \in \mathcal{S}_{1,3,4,5}$, then $C_n$ is $\chi_S$-vertex-critical if and only if $n \leq 11 $ or $n \in \{15,  17, 23\}$.
		
	\item[$(v)$] 	If $S \in \mathcal{S}_{1,3,5,5}$, then $C_n$ is $\chi_S$-vertex-critical if and only if $n \leq 10 $ or $n$ is odd.

      \item[$(vi)$] If $S \in \mathcal{S}_{1,2,\overline{4},7} \cup \cS_{1,3,\overline{4},7} \cup \cS_{1,3,\overline{5}, 6}$, then $C_n$ is $\chi_S$-vertex-critical if and only if 
$n \not\equiv 0 \pmod{8}$ or $n = 8$. 
\end{itemize}
    \end{theorem} 
    \begin{proof}
        All packing sequences $S=(1, s_2, s_3, s_4, \dots )$ considered in the statement satisfy $s_2 \in \{2,3\}$ and $s_3,s_4 \in \{4,5,6,7\}$.
         Thus  $S \in \cS(4)$ and Theorem~\ref{thm:path} shows
        \begin{equation} \label{eq:vertex-crit-1}
            \chi_S(P_n)= \min \{4,\lfloor \log_2(n) \rfloor +1\}
     \end{equation}
     for every $n \ge 1$. 

        Let $n \ge 3$. A cycle $C_n$ is $\chi_S$-vertex-critical if and only if $\chi_S(P_{n-1}) < \chi_S(C_n)$. We know that $\chi_S(P_{n-1}) \le \chi_S (P_n) \le \chi_S(C_n)$ is true for every packing sequence $S$. Consequently, $C_n$ is $\chi_S$-vertex-critical if and only if at least one of (a) $\chi_S(P_{n-1}) < \chi_S (P_n)$ and (b) $\chi_S (P_n) < \chi_S(C_n)$ holds. By the formula in~\eqref{eq:vertex-crit-1}, the strict inequality (a) holds if and only if $n \in \{4,8\}$. By definition, (b) is equivalent to $C_n$ being a $\chi_S$-critical cycle. 

        Using these facts, Theorem~\ref{thm-4.1} (i) and Theorem~\ref{thm-4.2} (i) imply part (i) in our theorem; parts (ii) and (iii) follow from Theorem~\ref{thm-4.1}; parts (iv) and (v) follow from Theorem~\ref{thm-4.2}; and part (vi) follows from Theorem~\ref{thm-4.3}.  
         \end{proof}

         Note that $\chi_S$-criticality and $\chi_S$-vertex-criticality coincide for paths by Proposition~\ref{prop:critical-path}. In contrast, this equivalence does not hold for cycles, as shown in the previous result.
\section*{Acknowledgements}

This work was supported by the Slovenian Research and Innovation Agency (ARIS) under the grants P1-0297, N1-0355, J1-70045, and by the Scientific and Technological Research Council of Türkiye (TÜBİTAK) under grant no.\ 124F114.


\end{document}